\documentclass[reqno]{amsart}

\usepackage{amsmath,mathrsfs,amsfonts,amssymb,latexsym}
\usepackage{enumerate}
\usepackage{hyperref}

\newtheorem{theorem}{Theorem}[section]
\newtheorem{lemma}[theorem]{Lemma}
\newtheorem{proposition}[theorem]{Proposition}
\newtheorem{corollary}[theorem]{Corollary}

\theoremstyle{definition}

\newtheorem{example}[theorem]{Example}
\newtheorem{problem}[theorem]{Problem}

\theoremstyle{remark}
\newtheorem{remark}[theorem]{Remark}

\newcommand{\abs}[1]{\left| {#1} \right|}

\newcommand{\bbZ}[0]{\mathbb{Z}}
\newcommand{\bbN}[0]{\mathbb{N}}

\newcommand{\calB}[0]{\mathcal{B}}
\newcommand{\calE}[0]{\mathcal{E}}
\newcommand{\calN}[0]{\mathcal{N}}
\newcommand{\floor}[1]{\left\lfloor{#1}\right\rfloor}

\newcommand{\scrD}[0]{\mathscr{D}}

\begin{document}

\title[Existentially Closed Graphs Arising from Block Designs]{Further Results on Existentially Closed Graphs Arising from Block Designs}


\author{Xiao-Nan Lu }


\address{Department of Industrial Administration,
Faculty of Science and Technology, Tokyo University of Science. 
2641 Yamazaki, Noda-shi, Chiba 278-8510, Japan.}

\email{lu@rs.tus.ac.jp}   

\thanks{This work was supported in part by JSPS under Grant-in-Aid for Scientific Research (B) No.~15H03636 and No.~18H01133.}

\subjclass[2010]{05B07, 05B05, 05C75}

\keywords{Existential closure, 
Block intersection graph,
Pairwise balanced design,
Triple system,
Steiner quadruple system}

\begin{abstract}
A graph is $n$-existentially closed ($n$-e.c.) 
if for any disjoint subsets $A$, $B$ of vertices with 
 $\abs{A \cup B}=n$, there is a vertex $z \notin A \cup B$ adjacent to every vertex of $A$ and no vertex of $B$. 
For a block design with block set $\calB$, its block intersection graph 
is the graph whose vertex set is $\calB$ and two vertices (blocks) are adjacent
 if they have non-empty intersection.

In this paper, we investigate the block intersection graphs of pairwise balanced designs, 
and propose a sufficient condition for such graphs to be $2$-e.c. 
In particular, we study the  $\lambda$-fold triple systems with $\lambda \ge 2$ and determine for  
which parameters their block intersection graphs are $1$- or $2$-e.c. 
Moreover, for Steiner quadruple systems, 
the block intersection graphs and their analogue called $\{1\}$-block intersection graphs
are investigated, and the necessary and sufficient conditions 
for such graphs to be $2$-e.c. are established. 
\end{abstract}

\maketitle

\section{Introduction}
\label{intro}

Throughout this paper, all graphs (except mentioned specially) are assumed to be finite, undirected, loopless, and 
without multiple edges. 
For any positive integer $v$, we use $[v]$ to denote the set $\{1,2, \ldots, v\}$. 
Given set $X$ and positive integer $t$, let $\binom{X}{t}$  denote 
the collection of all the subsets of cardinality $t$ ($t$-subsets for short) of $X$.  
Let $\bbZ$, $\bbN$ denote the sets of integers and positive integers, respectively.  

\subsection{Existentially closed graphs and existential closure number}

Let $n$ be a positive integer. 
A graph is  \emph{$n$-existentially closed}
($n$-e.c.)
 if for any pair of disjoint subsets $A$, $B$ of the vertex set such that  
 $\abs{A \cup B}=n$ (one of $A$ and $B$ may be empty), 
 there is a vertex $z \notin A \cup B$ adjacent to every vertex of $A$ and no vertex of $B$. 
 Clearly, 
 an $n$-e.c. graph is also $(n-1)$-e.c. for any $n \ge 2$.
The \emph{existential closure number} 
of  graph $\Gamma$, denoted by 
$\Xi(\Gamma)$,  is
the largest non-negative integer $n$ such that 
$\Gamma$ is $n$-e.c. 
For the sake of convenience,  $\Xi(\Gamma)$ is defined to equal $0$ if $\Gamma$ is not $1$-e.c. 
In this case, one can easily obtain the following proposition and the proof is omitted.

\begin{proposition}\label{prop:0ec}
Let $\Gamma$ be a connected graph on $b$ vertices. 
Then $\Xi(\Gamma) = 0$ if and only if the maximum degree $\Delta(\Gamma)$ of $\Gamma$ is $b-1$. 
\end{proposition}

As a kind of adjacency property of graphs, 
the existential closure property can be viewed as a finite analogue of the
``universal'' property of the countable random graph (see \cite{cameron1997random}). 
It is shown by Erd{\H{o}}s and R\'{e}nyi~\cite{erdos1963asymmetric} that
almost all finite graphs are $n$-e.c. for any fixed value  of $n \ge 1$.
However, explicit constructions for $n$-e.c. finite graphs are not extensively studied, where 
most are developed by employing 
algebra and/or number theory (see \cite{bonato2009search,bonato2001adjacency}). 
It seems challenging to establish new constructions from combinatorics straightforwardly. 

Remarkably, there are a few known constructions of $n$-e.c. graphs with $n \ge 3$ from   
combinatorial structures 
that are commonly used in Design Theory, 
for example, Paley graphs \cite{ananchuen1993adjacency,blass1981paley,bollobas1981graphs},
finite affine planes \cite{baker2003graphs,baker2008graphs}, and Hadamard matrices \cite{bonato2001hadamard}. 
It is interesting to investigate the adjacency properties of  graph structures directly obtained from designs,
 which are  naturally related to  the problem of design configurations (i.e., the inner structures of designs). 
In the next section, we will review basic concepts on block designs and their block intersection graphs.

For details on the countably infinite random graph, the reader is referred to \cite{cameron1997random}. 
For additional information on $n$-e.c. graphs, the interested reader is referred 
 to \cite{bonato2009search,bonato2001adjacency} and their references.

\subsection{Block designs and their block intersection graphs}

Let $V$ be a finite set and  let $\calB$ be a collection (multiset) of 
 subsets of $V$. 
Then,  $\scrD = (V, \calB)$ is a \emph{set system} (essentially, a hypergraph) and we call the subsets of $\calB$ \emph{blocks}.
Regarding the  blocks of $\calB$ as  vertices and 
joining an edge between two vertices if they have non-empty intersection, 
we obtain the \emph{block intersection graph} (BIG) of $\scrD$, denoted by $\Gamma_{\scrD}$.
By regarding set systems as hypergraphs, 
 that block intersection graphs  are also 
  known as \emph{line graphs} of hypergraphs.

More generally, let $S$ be a  finite set of non-negative integers. 
The \emph{$S$-block intersection graph} ($S$-BIG) is defined by
$\Gamma^{S}_{\scrD} = (\calB, \calE)$  with
\[
\calE = \left\{ \{B_1, B_2\} \in \binom{\calB}{2} :  \abs{B_1 \cap B_2} \in S \right\}. 
\]

Let $K$ be a subset of positive integers (called \emph{block size})
 and let $\lambda$ be a positive integer (called \emph{index}). 
A set system  $\scrD = (V, \calB)$ is a \emph{pairwise balanced design} (PBD)
 with parameters $(v, K, \lambda)$, or simply a $(v, K, \lambda)$-PBD, 
if (1) $\abs{V} = v$, (2) $\abs{B} \in K$ for every $B \in \calB$, and 
(3) each pair of distinct points of $V$ occurs in exactly $\lambda$ blocks of $\calB$.
In particular, for a positive integer $k$, 
a $(v, \{k\}, \lambda)$-PBD is usually called 
a \emph{balanced incomplete block design} (BIBD) with parameters $(v, k, \lambda)$, 
also known as a $2$-$(v, k, \lambda)$ design. 

More generally, for positive integers $t, v, k, \lambda$,  
a set system $\scrD = (V, \calB)$ is a \emph{$t$-$(v, k, \lambda)$ design},
if  (1) $\abs{V} = v$, (2) $\abs{B} = k$ for every $B \in \calB$, and 
(3) each $t$-subset of $V$ occurs in exactly $\lambda$ blocks of $\calB$.

A $t$-design with $\lambda=1$ is well known as a \emph{Steiner system} or a \emph{Steiner $t$-design}.
In particular, a $2$-$(v, 3, 1)$ design (resp., a $3$-$(v, 4, 1)$ design)
is usually called a \emph{Steiner triple system} (resp., a \emph{Steiner quadruple system})
of order $v$, denoted by STS$(v)$ (resp., SQS$(v)$). 
Moreover, $t$-designs with $\lambda \ge 2$ are called \emph{$\lambda$-fold} $t$-designs.

A $t$-$(v, k, \lambda)$ design is also an $h$-$(v, k, \lambda_h)$ design
with $\lambda_h = \left. \lambda \binom{v-h}{t-h} \middle/  \binom{k-h}{t-h} \right.$
for each $1 \le h \le t-1$. 
Basically, every $t$-design ($t \ge 3$) can be treated as a $2$-design on the same point set 
with the same block size and larger index. 

In a $2$-$(v, k, \lambda)$ design $\scrD = (V, \calB)$, 
the number $r$ of blocks that containing 
a given point $x \in V$ 
is called the \emph{replication number},
which  is independent of how $x$ is chosen.  
The number of blocks $b = \abs{\calB}$ and the replication number $r$
satisfy the relation that $vr = bk$. 
By Proposition~\ref{prop:0ec}, the BIG $\Gamma_{\scrD}$ has existential closure number $0$
if and only if there exists some block $B$ having non-empty intersection with 
every block in $\calB$. 
Hence, there must be $r=b$, which implies $k=v$. 
Therefore, $\Xi(\Gamma_\scrD) \ge 1$ whenever $v > k$.

Other required preliminaries of block designs will be given before discussion. 
For more details on design theory, the reader is referred to \cite{beth1999design,colbourn2007handbook}. 
For further information on triple systems and (Steiner) quadruple systems,
which are the main subjects in Sections~\ref{sec:PBD} and \ref{sec:SQS}, 
the reader is referred to  
the monograph \cite{colbourn1999triple} and the survey papers \cite{hartman1992steiner,lindner1978steiner},
respectively.

The rest of this paper is organized as follows. In Section~\ref{sec:0}, we shall 
summarize some important results on existentially closed BIGs of BIBDs. 
In Section~\ref{sec:PBD}, 
a general result on the existential closure  number of BIGs of PBDs
will be given in Theorem~\ref{thm:BIG_PBD}, 
which covers the previous results for Steiner $2$-designs and $2$-fold triple systems 
due to McKay and Pike~\cite{mckay2007existentially}.   
In addition, $\lambda$-fold triple systems
(i.e., $2$-$(v,3, \lambda)$ designs) will be discussed in detail in the last part of Section~\ref{sec:PBD}.   
Finally, in Section~\ref{sec:SQS}, we shall focus on quadruple systems (i.e., $3$-$(v,4,\lambda)$
designs) and derive the necessary and sufficient conditions for BIGs and $\{1\}$-BIGs
of Steiner quadruple systems to be $2$-e.c.

\section{Known results on BIGs of BIBDs}
\label{sec:0}

First, we summarize some known results on  existentially closed BIGs
arising from BIBDs (i.e., $2$-$(v,k,\lambda)$ designs) 
due to McKay and Pike~\cite{mckay2007existentially}.

\begin{theorem}[\cite{mckay2007existentially}]
\label{thm:n-e.c._basic}
Let $n$ be a positive integer and let $\scrD$ be a $2$-$(v, k ,\lambda)$ design.  
If $\Xi(\Gamma_{\scrD}) \ge n$, then 
\begin{enumerate}[(i)]
\item
\label{thm:n-e.c._basic_1}
$v \ge (n+1)k$.
\item
\label{thm:n-e.c._basic_2}
$n \le k$ if $\lambda =1$.
\item
\label{thm:n-e.c._basic_3}
$n \le \floor{(k+1)/2}$ if $\lambda \ge 2$.
\item
\label{thm:n-e.c._basic_4}
$\scrD$ is simple if $n \ge 2$.
\end{enumerate}
\end{theorem}

\begin{theorem}[\cite{mckay2007existentially} Theorem~2]
\label{thm:2-e.c.-lower}
Let $\scrD = (V, \calB)$ be a $2$-$(v, k ,1)$ design with $k \ge 3$. 
Then $\Xi(\Gamma_\scrD) \ge 2$ if and only if $v \ge k^2+k-1$.
\end{theorem}

For the BIGs having existential closure number at least $3$,
 the following upper bound is known.

\begin{theorem}[\cite{mckay2007existentially} Theorem~3]
\label{thm:3-e.c.-upper}
Let $n$ be an integer with $n \ge 3$ and 
 $\scrD = (V, \calB)$ be a $2$-$(v, k ,\lambda)$ design with $k \ge 3$.  
 If $\Xi(\Gamma_\scrD) \ge n $, then $v \le \lambda k^4 - \lambda nk^3 + (\lambda+1)(n-1) k^2 - nk +k +1$.
\end{theorem}

In particular, 
for  STS$(v)$ (i.e., $2$-$(v, 3, 1)$ designs),
the following can be directly obtained from Theorems~\ref{thm:n-e.c._basic} and \ref{thm:2-e.c.-lower}
 (see also \cite{forbes2005steiner}).

\begin{corollary}[\cite{forbes2005steiner}]
\label{cor:STS_2ec}
Let $\scrD$ be an STS$(v)$.  
Then $\Xi(\Gamma_\scrD) \ge 2$  if and only if  $v \ge 13$. 
Moreover, if $\Xi(\Gamma_\scrD) \ge 3$ then $v \le 31$.
\end{corollary}

Forbes, Grannell, and Griggs~\cite{forbes2005steiner} carefully investigated
if there exists an STS$(v)$ with $13 \le v \le 31$, such that 
its BIG is $3$-e.c. 
It is proved that $\Xi(\Gamma_\scrD) = 2$ 
if $\scrD$ is an STS$(v)$ with $v \in \{13, 15, 25, 27, 31\}$.
Therefore,  $\Xi(\Gamma_\scrD)$ is  possibly $3$ only if  
 $\scrD$ is an STS$(19)$ or an STS$(21)$.  
By systematically examining all STS($19$) 
it is shown in~\cite{colbourn2010properties} that  there are the only two
STS$(19)$ having $3$-e.c. BIGs among all STS$(19)$.

The notion of (finite) designs can be generalized to infinite designs, whose BIGs have also been investigated.
For a $t$-design having finite block size and finite index,
 it is shown in~\cite{pike2011existential} that 
 if the BIG is $n$-e.c., then $n \le t + 1$. 
However, 
 it is shown in~\cite{horsley2011existential} that 
there exists a $t$-design with infinite block size, 
 such that its BIG  is $n$-e.c. for every non-negative integer $n$.
In addition, it is noticed recently in \cite{kikyo2017kohler}
that there exists a $2$-design  with infinite block size, 
such that its BIG has existential closure number $0$. 

In summary, the BIGs of various kinds of infinite designs behave quite differently 
in the aspect of their existential closure properties. 
Although it seems that there is a significant gap between finite and infinite designs, 
  it is still interesting to clarify the exact existential closure number for designs that are shown 
to be  (possibly) not $2$-e.c.  

\section{Existential closure number of BIGs of PBDs}
\label{sec:PBD}

In this section, we assume that $\scrD= (V, \calB)$ is a $(v, K, \lambda)$-PBD. 
Notice that $\Gamma_{\scrD}$ is a non-regular graph if $\abs{K} \ge 2$. 
The design $\scrD$ is said to be \emph{$1$-cover-free} (also called \emph{reduced} for hypergraphs)
if for any block $B \in \calB$, there does not exist another block which is a superset of $B$. 
In particular,
it is equivalent to saying that 
$\scrD$ is \emph{simple}  when $\scrD$ is a BIBD (i.e., $\abs{K}=1$). 

\begin{proposition}
Let $\scrD=(V, \calB)$ be a $(v, K, \lambda)$-PBD with $\lambda \ge 2$.
If $\Xi(\Gamma_\scrD) \ge 2$, then  $\scrD$ must be $1$-cover-free.
\end{proposition}
\begin{proof}
Assume there exists $B^\prime \in \calB \setminus \{B\}$
so that $B \subseteq B^\prime$ for some given  $B \in \calB$.
Then it is impossible to find a block in $\calB \setminus \{B, B^\prime\}$, say $B^\ast$, 
such that $B^\ast \cap B \ne \emptyset$ but  $B^\ast \cap B^\prime = \emptyset$,
which is equivalent to saying that $\Gamma_\scrD$ is not $2$-e.c.
\end{proof}

In PBDs, 
the replication numbers (i.e., the number of blocks containing a certain point)  may vary for different points. 
The following proposition may be well known for design theorists. 
However, we give a proof here for completeness.  

\begin{proposition}
\label{prop:basic_prop_G_PBD}
Suppose $\scrD=(V, \calB)$ is a $(v, K, \lambda)$-PBD. 
Let $k_{\min} = \min K$ and $k_{\max} = \max K$ be the smallest and the largest block sizes of $\scrD$, respectively.  
For any point $i \in V$,  let $r_i$ denote the number of blocks containing $i$ in $\calB$.
Then,  
$\lambda (v-1)/({k_{\max}-1}) \le r_i \le \lambda (v-1)/({k_{\min}-1})$   { for every } 
$i \in V$. 
\end{proposition}
\begin{proof}
Let $b := \abs{\calB}$ be the number of blocks and 
consider the incidence matrix $M = (m_{x, B})_{v \times b}$ of $\scrD$, defined by
\[
m_{x,B} = 
\begin{cases}
1, & \text{ if } x \in B, \\ 
0, & \text{ otherwise,} 
\end{cases}
\]
where rows and columns are indexed by $x \in V$ and $B \in \calB$, respectively.  
One can rearrange the rows and the columns of $M$ such that the first row is indexed by point $i$ which occurs in the first $r_i$ blocks. 
Then, we have
\[
M= 
\begin{bmatrix}
1 \; \cdots \; 1 & 0 \; \cdots \; 0  \\ 
M_1 &  M_2  
\end{bmatrix},
\]
where $M_1$ is a $(v-1) \times r_i$ sub-matrix. 
In each row of $M_1$, the number of $1$'s is $\lambda$. 
While, in each column of $M_1$, the number of $1$'s is equal to the corresponding block size 
minus $1$. 
Therefore, we have $\lambda(v-1) \le r_i(k_{\max} - 1)$
and $\lambda(v-1) \ge r_i(k_{\min} - 1)$, which prove the claim.
\end{proof}

\begin{theorem}\label{thm:BIG_PBD}
Let  $\scrD$ be a $1$-cover-free $(v, K, \lambda)$-PBD such that 
\begin{equation}\label{eq:thm_PBD}
v >  2k_{\max}( k_{\max} -1)+1,
\end{equation}
where $k_{\max} = \max K \ge 2$. Then $\Xi(\Gamma_\scrD) \ge 2$.
\end{theorem}

\begin{proof}
The basic idea is  similar to that of the proof for $k=3$ and $\lambda=2$
 (see~\cite{mckay2007existentially} Lemma~5). 
Without loss of generality, suppose $\scrD = ([v], \calB)$. 
For any block $B$, let $\calN(B)$ 
denote the neighborhood of $B$  in $\Gamma_\scrD$,
i.e., $\calN(B) = \{ T \in \calB \setminus \{B\} : T \cap B \ne \emptyset \}$. 

It follows from the assumption  \eqref{eq:thm_PBD} and 
Proposition~\ref{prop:basic_prop_G_PBD} 
that 
\begin{equation}
\label{eq:bound_PBD_prf_1}
r_i > 2\lambda k_{\max}  \quad \text{ for every } i \in V.
\end{equation}

Let $B_1$ and $B_2$ be two distinct blocks of $\calB$, 
where $\abs{B_1}=k$ and $\abs{B_2}=k^\prime$. 
Here, $k, k^\prime \in K$  are not necessarily distinct.
Let $\ell := \abs{B_1 \cap B_2}$. 
Since $\scrD$ is $1$-cover-free, $B_1$ and $B_2$ have 
at most $\min\{ k, k^\prime \}-1$ points in common, 
i.e., $\ell \in \{0,1, \ldots, \min\{ k, k^\prime \}-1\}$.
Without loss of generality, let $B_1=\{1, 2, \ldots, k\}$ and 
$B_2=\{k- \ell+1, k- \ell+2, \ldots, k- \ell+k^\prime\}$.

It suffices to show the existence of the following three blocks in $\calB \setminus \{B_1, B_2\}$:
\begin{enumerate}[(i)]
\item $T_{11}$, such that  $T_{11} \cap B_1 \neq \emptyset$ and $T_{11} \cap B_2 \neq \emptyset$; 
\item $T_{10}$, such that $T_{10} \cap B_1 \ne \emptyset$ and $T_{10} \cap B_2 = \emptyset$, or equivalently 
\item[(ii${}^\prime$)] $T_{01}$, such that $T_{01} \cap B_1 = \emptyset$ and $T_{01} \cap B_2 \ne \emptyset$; 
\item $T_{00}$, such that $T_{00} \cap B_1 = T_{00} \cap B_2 =  \emptyset$.
\end{enumerate} 

\paragraph{Case (i).} First,  $T_{11}$ can be any one of the $\lambda$ blocks containing $\{1, k-\ell + k^\prime \}$. 

\paragraph{Case (ii).} 
Besides $B_1$, there are $r_1-1$ blocks containing the point $1$. 
Among these $r_1-1$ blocks, there are at most $\lambda k^{\prime}$ blocks 
containing a pair of the form $\{1, h\}$ with $h \in B_2$. 
Hence,  in $\calN(B_1)$, there are at least $r_1 -1 - \lambda k^{\prime}$ blocks,
which are candidates of $T_{10}$, 
disjoint from $B_2$, where 
$r_1 -1 - \lambda k^{\prime} \ge r_1 -1- \lambda k_{\max} > \lambda k_{\max} -1 > 0$ by \eqref{eq:bound_PBD_prf_1}. 

\paragraph{Case (iii).} 
Take a point $x \notin B_1 \cup B_2$. 
Among all $r_x$ blocks containing $x$, 
there are at most $\lambda(k+k^\prime-\ell)$ blocks  containing a pair of the form 
$\{x, h\}$ with $h \in B_1 \cup B_2$. 
It is guaranteed by \eqref{eq:bound_PBD_prf_1} that 
$r_x \ge 2 \lambda  k_{\max} + 1 \ge 
\lambda (k + k^\prime -\ell) + 1$.
So there must exist $T_{00} \in \calB$ such that  $x \in T_{00}$ and
$T_{00} \cap (B_1 \cup B_2) = \emptyset$.  
\end{proof}

\begin{remark}\label{rem:BIG_PBD}
If  $\scrD$ is a $1$-cover-free $(v, K, \lambda)$-PBD with $k_{\max} \ge 3$, 
$\lambda \ge 2$, and $v \ge k_{\max}^2$, 
then the blocks $T_{11}$ and $T_{10}$ required in the proof of Theorem~\ref{thm:BIG_PBD}
must exist. 
In this case, the replication number $r_i$ satisfies $r_i \ge  \lambda (v-1)/(k_{\max}-1) \ge \lambda(k_{\max}+1) > \lambda k_{\max} + 1$. 
\end{remark}

\begin{corollary}\label{cor:TS_lambda_ge_13}
Let $k$, $\lambda$ be  positive integers with $k \ge 3$ and 
let $\scrD$ be a simple $2$-$(v, k, \lambda)$ design with 
$v > 2k^2-2k+1$. 
Then  $\Xi(\Gamma_\scrD) \ge 2$. 
\end{corollary}

 McKay and  Pike (\cite{mckay2007existentially} Question~2) asked 
whether there exists a family of $2$-$(v,k,\lambda)$ designs $\scrD$ so that 
$\Xi(\Gamma_\scrD) \ge \floor{(k+1)/2}$ for $\lambda \ge 2$ and $k \ge 3$.
Now Corollary~\ref{cor:TS_lambda_ge_13} gives  an affirmative answer for  $k\in \{3, 4\}$ and any $\lambda \in \bbN$.
Combining with Theorem~\ref{thm:n-e.c._basic}~\eqref{thm:n-e.c._basic_3}, 
we have the following:

\begin{corollary}\label{cor:TS_lambda}
Let $\scrD$ be a simple $2$-$(v, 3, \lambda)$ design with $\lambda \ge 2$. 
Then  $\Xi(\Gamma_\scrD) = 2$ whenever  $v > 13$. 
\end{corollary}

Now we use the shorter notation TS$(v, \lambda)$ for 
a $\lambda$-fold triple system of order $v$ (i.e., 
a $2$-$(v, 3, \lambda)$ design). 
It is proved in \cite{dehon1983existence} (see also \cite{colbourn2007handbook}~\S II.2.64) that 
a simple TS$(v, \lambda)$ exists if and only if 
\begin{equation}\label{eq:simple_TS}
\lambda \le v-2 \qquad \text{ and } \qquad \gcd(v-2,6) \mid \lambda.
\end{equation}
A positive integer $v$ satisfying the conditions \eqref{eq:simple_TS}
is said to be \emph{$\lambda$-admissible}.  
The only simple TS$(v, v-2)$ is isomorphic to $([v], \binom{[v]}{3})$
and its BIG is clearly $2$-e.c. whenever $v \ge 9$. 
 Now it remains to consider simple TS$(v, \lambda)$
with $\lambda$-admissible integer $v$ such that $9 \le v \le 13$ and $2 \le \lambda \le v-3$.  
 
It is shown in~\cite{mckay2007existentially} that  
$\Xi(\Gamma_\scrD) = 2$ if $\scrD$ is a TS$(v, 2)$ with $v \ge 13$; 
$\Xi(\Gamma_\scrD) = 1$ if $\scrD$ is a TS$(v, 2)$  with $9 \le v \le 10$.  
Here we extend the above results for TS$(v, \lambda)$ with $\lambda \ge 2$. 

\begin{lemma}
\label{lem:TS_lambda_13}
Let $\scrD$ be a simple TS$(13, \lambda)$ with $\lambda \ge 2$. 
If  $\lambda \ne 4$, then $\Xi(\Gamma_\scrD) = 2$. 
If  $\lambda = 4$,  $\Xi(\Gamma_\scrD) = 2$ if and only if
$\scrD$ does not contain a simple TS$(6, 4)$ as a sub-design.
\end{lemma}
\begin{proof}
First, it follows from Theorem~\ref{thm:n-e.c._basic}~\eqref{thm:n-e.c._basic_3} that 
$\Xi(\Gamma_\scrD) \le 2$. 
Next, 
as shown in the proof of Theorem~\ref{thm:BIG_PBD},
for $\scrD = (V, \calB)$ a TS$(13, \lambda)$ with $V= \{1,2, \ldots, 13\}$, 
it suffices to find a block disjoint from $B_1 \cup B_2$,
where $B_1= \{1,2,3\}$ and $B_2 =\{4,5,6\}$. 

Assume each $T \in \calB$ has non-empty intersection with $B_1 \cup B_2$.
Let $H=B_1 \cup B_2$  and $W= V \setminus H$.
Then there are exactly $\lambda{\abs{W} \choose 2} = 21\lambda$ blocks of the form
$\{h, w_1, w_2\}$ with $h \in H$ and $w_1, w_2 \in W$.
None of these $21\lambda$ blocks contains a pair from ${H \choose 2}$.
Accordingly, the remaining $\abs{\calB} - 21\lambda = 5 \lambda$ blocks 
(including $B_1$ and $B_2$) should only consist of the points from $H$, 
which means that  these $5 \lambda$ blocks  form a simple TS$(6,  \lambda)$  
on $H$.
Note that a simple TS$(6, \lambda)$ exists if and only if $\lambda \in \{2,4\}$ (cf. Eqn.~\eqref{eq:simple_TS}). 
We conclude that there must exist some $T \in \calB$ disjoint with $B_1 \cup B_2$,
which guarantees that $\Xi(\Gamma_\scrD) \ge 2$ if  
$\lambda \in \{3\} \cup \{5,6, \ldots, 12\}$.

Furthermore, since the only TS$(6, 2)$ does not have a parallel class 
(i.e., a subset of blocks that partitions the point set $V$) 
of the form $\{B_1, B_2\}$,
so  $\lambda =2$ can be eliminated.
Similarly, for $\lambda=4$,  if $\scrD$ does not contain   TS$(6,4)$ as a sub-design, 
we can also conclude that $\Xi(\Gamma_\scrD) \ge 2$. 

Conversely, let $\scrD$ be a simple TS$(13, 4)$ having a sub-TS$(6,4)$ on $H=\{1,2, \ldots, 6\}$.  
There are $b_1 := 20$ triples in the sub-TS$(6,4)$, containing all the pairs of $\binom{H}{2}$ four times. 
So every pair of the form $\{h, w\}$ should be contained in some triple of the 
form $\{h, w, w^\prime\}$, where 
$h \in H$ and 
$w, w^\prime \in V \setminus H$. 
The total number of such pairs is
$\lambda \cdot \abs{H} \cdot \abs{ V \setminus H} = 168$,
and hence the number of corresponding triples should be $b_2:=84$. 
Note that the total number of triples in a TS$(13, 4)$ is $104$ which is equal to $b_1 + b_2$. 
Therefore, 
the TS$(13,4)$ does not contain any triple in $\binom{V \setminus H}{3}$,
which implies that there does not exist a triple disjoint from $H$ in $\scrD$. 
This completes the proof. 
\end{proof}

\begin{remark}
The existence of a simple TS$(13, 4)$ containing a simple TS$(6,4)$ as a sub-design is 
guaranteed by a theorem due to  Shen~\cite{shen1992embeddings} (see also \cite{colbourn1999triple} Theorem~6.22), 
which states that 
there exists a simple TS$(v, \lambda)$ having a simple sub-TS$(w, \lambda)$ 
if and only if $v$ is $\lambda$-admissible,   $w$ is $\lambda$-admissible, 
$v \ge 2w+1$, and $\lambda \le w-2$. 
\end{remark}

\begin{remark}
There are four non-isomorphic TS$(6,4)$, in which only one is simple  
(see \cite{kaski2006classification} and its supplementary material \cite{kaski2006classificationSU}). 
\end{remark}

It is known that the number of non-isomorphic TS$(13,4)$  is greater than $10^8$ 
(see~\cite{mathon1985tables} page~285; see also~\cite{colbourn2007handbook} page~42). 
Here,  we give an example of a simple TS$(13,4)$ whose BIG has existential closure number $2$. 


\begin{example}\label{ex:TS_13_4_2}
Let $V = \bbZ_{13} = \bbZ/13\bbZ = \{0,1, \ldots, 12\}$ and $\calB= \{  \{1,3,9\} + i, \{2,5,6\}+i : i \in \bbZ_{13}  \}$. 
Then $(V, \calB)$ is the well known \emph{Netto system} on $13$ points. 
Define three permutations on  $V$ as follows: 
\begin{align*}
\sigma_1 &= 
\left( \begin{array}{@{}*{13}{c}@{}}
0 & 1 & 2 & 3 & 4 & 5 & 6 & 7 & 8 & 9 & 10 & 11 & 12 \\
1 & 6 & 7 & 2 & 4 & 10 & 3 & 5 & 8 & 11 & 9 & 0 & 12
\end{array} 
\right), \\
\sigma_2 &= 
\left( \begin{array}{@{}*{13}{c}@{}}
0 & 1 & 2 & 3 & 4 & 5 & 6 & 7 & 8 & 9 & 10 & 11 & 12 \\
4 & 5 & 11 & 2 & 1 & 6 & 9 & 10 & 7 & 12 & 3 & 0 & 8
\end{array} 
\right), \\
\sigma_3 &= 
\left( \begin{array}{@{}*{13}{c}@{}}
0 & 1 & 2 & 3 & 4 & 5 & 6 & 7 & 8 & 9 & 10 & 11 & 12 \\
12 & 1 & 2 & 7 & 4 & 10 & 9 & 3 & 5 & 11 & 8 & 6 & 0 
\end{array} 
\right). 
\end{align*}
Then, $(V, \calB \cup \sigma_1(\calB) \cup \sigma_2(\calB) \cup \sigma_3(\calB))$
is a simple  TS$(13,4)$ and its BIG is $2$-e.c. 
\end{example}

It is reported in \cite{mckay2007existentially}  that 
among all the 
$88,616,310$ non-isomorphic simple TS$(12,2)$,
the vast majority have 2-e.c. BIGs, whereas as only $286,962$ fail to be 2-e.c. 

\begin{remark}\label{rem:TS_12_lambda_2}
We examined all $74,700$ non-isomorphic resolutions of resolvable TS$(12, 2)$, where 
$49,533$ are simple, and  $48,588$ have $2$-e.c. BIGs. 
The BIGs of the corresponding supplementary TS$(12, 8)$ are all $2$-e.c. 
(See \cite{kaski2006classification} and its supplementary material \cite{kaski2006classificationSU} for the data of designs). 
\end{remark}

As in Example~\ref{ex:TS_13_4_2},  
by properly combining two (resp., three) TS$(12,2)$, 
one can generate a simple TS$(12,4)$ (resp., TS$(12,6)$), such that its BIG 
has existential closure number  $2$.

\begin{example}\label{ex:TS_11_3_2}
Let $V = \bbZ_{11} = \bbZ/11\bbZ  = \{0,1, \ldots, 10\}$ and 
$\calB= \{  \{0, j, 2j\} + i :  i \in \bbZ_{11}, 1 \le j \le 5  \}$. 
Then $\scrD = (V, \calB)$ is a simple TS$(11,3)$ with $\Xi(\Gamma_\scrD) = 2$. 
The supplementary design
$\tilde{\scrD} =(V,  \binom{V}{3} \setminus \calB)$ is a simple TS$(11,6)$
with   $\Xi(\Gamma_{\tilde{\scrD}}) = 2$. 
\end{example}

\begin{remark}
\label{rem:TS_lambda_10_6}
There are $960$ non-isomorphic TS$(10,2)$, in which $394$ are simple. 
For each simple TS$(10,2)$, say $\scrD=(V, \calB)$,  its supplementary design
$\tilde{\scrD} =(V,  \binom{V}{3} \setminus \calB)$, a simple TS$(10,6)$, 
is individually examined.
Only $10$ of them fail to be $2$-e.c. 
As in Example~\ref{ex:TS_13_4_2},  
by properly combining two TS$(10,2)$, 
one can generate a simple TS$(10,4)$, such that its BIG 
has existential closure number  $1$ or $2$.
(See \cite{kaski2006classification} and its supplementary material \cite{kaski2006classificationSU} for the data of designs). 
\end{remark}

\begin{lemma}\label{lem:TS_lambda_9}
Let $\scrD$ be a simple TS$(9,  \lambda)$ with $2 \le \lambda \le 6$.
Then $\Xi(\Gamma_\scrD) =1$. 
\end{lemma}
\begin{proof}
The replication number and the number of blocks of $\scrD = (V, \calB)$ are $r =4 \lambda$
and $b = 12\lambda$, respectively.
For any triple $B_1 \in \calB$, 
there are $3r- 3\lambda  = 9 \lambda$ other blocks having non-empty intersection with $B_1$.
Denote by $\calB^\prime$ the set of  blocks  disjoint from $B_1$. 
Then $\abs{\calB^\prime} = 3\lambda-1$. 
If $\Gamma_\scrD$ is $2$-e.c., for any $B_2 \in \calB^\prime$, 
there must exist a block disjoint from both $B_1$ and $B_2$,
which must be equal to $V \setminus (B_1 \cup B_2)$.
Therefore, $\calB^\prime$ should be able to be partitioned into disjoint pairs of triples.
This requires $3\lambda-1$ to be even.
Therefore, if $\Xi(\Gamma_\scrD) \ge 2$ then $\lambda$ is odd.

Next, suppose $\lambda \in \{3, 5\}$. 
With the help of computers, 
it is verified that $\Xi(\Gamma_\scrD) =1$ for these two cases.

There are $22,521$ non-isomorphic TS$(9,3)$, in which $332$ of them are simple. 
We individually examined these simple TS$(9,3)$, and found that 
none of them can give $2$-e.c. BIG. 

There are $36$ non-isomorphic TS$(9,2)$, in which $13$ of them are simple.
For each simple TS$(9,2)$, say $\scrD=(V, \calB)$,  its supplementary design
$\tilde{\scrD} := (V,  \binom{V}{3} \setminus \calB)$, is a simple TS$(9,5)$. 
All the thirteen simple TS$(9,5)$ are individually examined, and the BIGs failed to be $2$-e.c.

For the classification results of TS$(9,2)$ and TS$(9,3)$, 
the reader is referred to \cite{kaski2006classification} and its supplementary material~\cite{kaski2006classificationSU}. 
\end{proof}

We summarize in Table~\ref{tab:ec_num_TS}
the results of $\Xi(\Gamma_{\scrD})$ for 
$\scrD$ a simple TS$(v, \lambda)$ with $\lambda  \in \Lambda_v :=  \{ \lambda \in \bbN : 2 \le \lambda \le v-3, \gcd(v-2,6) \mid \lambda \}$. 
We finish the current section by posing the following remaining problem.

\begin{table}
\caption{Existential closure numbers of BIGs of TS$(v, \lambda)$} 
\label{tab:ec_num_TS}
\begin{tabular}{llllll}
\hline\noalign{\smallskip}
Order $v$ &  Index $\lambda$ 	  & $\Xi(\Gamma_\scrD)$ & References  \\
\noalign{\smallskip}\hline\noalign{\smallskip}
$v > 13$ & $\Lambda_v$  & $2$  & Corollary~\ref{cor:TS_lambda} \\  
$v = 13$ & $\Lambda_v \setminus \{4\}$  & $2$  &   Lemma~\ref{lem:TS_lambda_13} \\ 
 		& $\{4\}$  & $\{1,2\}$  & Lemma~\ref{lem:TS_lambda_13} and Example~\ref{ex:TS_13_4_2}
		\\   
$v = 12$ & $\{4,6,8\}$  & $2$?  &  Remark~\ref{rem:TS_12_lambda_2} \\ 
 		& $\{2\}$  & $\{1,2\}$ $^\natural$  &    \cite{mckay2007existentially} Theorem~6$^\dagger$   \\  
 $v = 11$ & $\{3,6\}$  & $2$?  & Example~\ref{ex:TS_11_3_2} \\  
 $v = 10$ &  $\{6\}$  & $\{1,2\}$  $^\flat$   & Remark~\ref{rem:TS_lambda_10_6}$^\dagger$ \\  
		&  $\{4\}$  & $\{1,2\}$   &  Remark~\ref{rem:TS_lambda_10_6}$^\dagger$ \\  
		& $\{2\}$  & $1$  &   \cite{mckay2007existentially}~Lemma~7$^\dagger$ \\  
$v = 9$  & $\Lambda_v$  & $1$  & Lemma~\ref{lem:TS_lambda_9}$^\dagger$  \\   
$v < 9$ & $\Lambda_v$  & $1$  &   \cite{mckay2007existentially} Lemma~2 
		(Theorem~\ref{thm:n-e.c._basic}~\eqref{thm:n-e.c._basic_1})\\ 
\noalign{\smallskip}\hline 
\end{tabular}
{ \footnotesize \\
$\dagger$: With the help of computers.  \\
$\natural$:  Among $88,616,310$ BIGs, only $286,962$ are not $2$-e.c. ($\approx 0.32\%$). \\
$\flat$:  Among $394$ BIGs, only $10$ are  not $2$-e.c.  ($\approx 2.54\%$).  \\
$?$: The existence of deigns having BIGs of existential closure number $1$ is unknown. 
}
\end{table}

\begin{problem}
Show the 
existence or non-existence of TS$(v, \lambda)$
with $(v, \lambda) \in \{(11,3), (11,6),  (12, 4), (12, 6), (12, 8) \}$ 
whose BIGs have existential closure number $1$. 
In fact, for showing the existence, it suffices to find such a  TS$(v, \lambda)$,
in which there are two blocks $B_1$ and $B_2$ 
such that $B_1 \cup B_2$ has non-empty intersection with every block. 
\end{problem}

\section{Existentially closed graphs arising from quadruple systems }
\label{sec:SQS}

\subsection{Block intersection graphs of quadruple systems }
\label{sec:SQS1}

Let $\scrD = (V, \calB)$ be a $t$-$(v, k, \lambda)$ design 
and let $M$ be an $m$-subset of $V$. 
For $0 \le i \le m$, the cardinality 
\[
\alpha_i := \big| \{ B \in \calB : \abs{B \cap M} = i \} \big| 
\]
 is said to be the \emph{intersection number} of $M$ in $\scrD$.
Mendelsohn~\cite{mendelsohn1971intersection} introduced this notion and 
studied  the case when $M$ is a block in $\calB$. 
A general result for arbitrary $M$ was given by  K\"{o}hler~\cite{kohler1989allgemeine};
See also Piotrowski~\cite{piotrowski1977notwendige} for an earlier result on  $\alpha_0$. 

\begin{theorem}[\cite{kohler1989allgemeine}]
\label{thm:intersection_num}
The following equation holds for $0 \le i \le t$:
\begin{equation*}\label{eq:bi_number_kolher} 
\alpha_i = 
\sum_{h=i}^{t} (-1)^{h+i} \binom{h}{i} \binom{m}{h} \lambda_h
	+ (-1)^{t+i+1} \sum_{h=t+1}^{m}  \binom{t-i-1}{h-t-1}  \binom{h}{i}  \alpha_{h},
\end{equation*} 
where 
\begin{equation*}\label{eq:t-des_1}
\lambda_h = \left. \lambda \binom{v-h}{t-h} \middle/  \binom{k-h}{t-h} \right..
\end{equation*}
\end{theorem}

\begin{lemma}\label{lem:QS}
Let $\scrD = (V, \calB)$ be a $\lambda$-fold quadruple system  (i.e., a  $3$-$(v, 4, \lambda)$
 design) of order $v \ge 16$.
 Then there exists a block in $\calB$ disjoint from
the union of any two given blocks. 
\end{lemma}
\begin{proof}
Let $M$ be the union of two blocks  
and  let $m = \abs{M}$. 
By setting $t=3$, $k=4$ in Theorem~\ref{thm:intersection_num} and 
noting that $\alpha_h = 0$ for each $h \ge k+1 = 5$,   we have
\begin{align*}
\alpha_0 &=  \alpha_{4} + 
\sum_{h=0}^{3} (-1)^{h}  \left.  \lambda  \binom{m}{h} \binom{v-h}{3-h} \middle/  \binom{4-h}{3-h} \right. \\
&= \alpha_{4} + \frac{\lambda}{24} (v - 2m) ( v^2 - 2mv  - 3v + 2m^2 + 2) \\
&\ge  \alpha_{4}, 
\end{align*}
whenever $m \ge 4$ and $v \ge 2m$. 
For any two blocks $B_1$ and $B_2$, we have  $4 \le m \le 8$
and $\alpha_4 \ge 2$. 
So, if $v \ge 16 \ge 2m$, 
 there must exist at least two blocks disjoint from 
$B_1 \cup B_2$. 

\end{proof}

It follows from Corollary~\ref{cor:TS_lambda_ge_13} that 
 $\Xi(\Gamma_\scrD) \ge 2$ if
 $\scrD$ is a simple $2$-$(v, 4, \lambda)$ design with $\lambda \ge 2$ and 
   $v > 25$. 
Theorem~\ref{thm:SQS_BIG} will give a more precise characterization 
 for the special case when
$\scrD$ is a Steiner quadruple system.  

\begin{theorem} \label{thm:SQS_BIG}
Let $\scrD$ be a simple $\lambda$-fold quadruple system of order $v$ (i.e., a  $3$-$(v, 4, \lambda)$
 design).
 Then $\Xi[\Gamma_{\scrD}] = 2$ whenever $v \ge 16$. 
In particular, if $\scrD$ is an SQS($v$) (i.e., $\lambda=1$), then 
$\Xi[\Gamma_{\scrD}] = 2$ if and only if $v \ge 16$. 
\end{theorem}
\begin{proof}
A $3$-$(v, 4, \lambda)$ design is also
a $2$-$(v, 4, \lambda (v-2)/2)$ design. 
By Theorem~\ref{thm:n-e.c._basic}~\eqref{thm:n-e.c._basic_3}, 
we have $\Xi[\Gamma_{\scrD}]  \le 2$.  
Moreover, it follows from Theorem~\ref{thm:n-e.c._basic}~\eqref{thm:n-e.c._basic_1}
that $v \ge 12$ is a lower bound so that $\Gamma_\scrD$ is $2$-e.c. 
Proceeding similarly to the proof of Theorem~\ref{thm:BIG_PBD}, 
it suffices to find  the blocks of Cases (i), (ii), and (iii).  
We notice that the blocks of Cases (i) and (ii) must exist whenever $v \ge 16$
(see also Remark~\ref{rem:BIG_PBD}).  
Moreover, Lemma~\ref{lem:QS} indicates that 
blocks of Cases (iii) exist if $v \ge 16$. 
The proof of the first statement is then completed. 

It is well known that an SQS$(v)$ exists if and only if $v \equiv 2$ or $4 \pmod{6}$
(see \cite{hanani1960quadruple}, also \cite{hartman1992steiner,lindner1978steiner} 
and \cite{colbourn2007handbook} \S II.5.5). 
If $\scrD$ is an SQS($v$), it remains to consider the case of $v = 14$. 
Actually, there are only four non-isomorphic SQS$(14)$  
(see \cite{mendelsohn1972},  also \cite{hartman1992steiner,lindner1978steiner}  and 
\cite{colbourn2007handbook} \S II.5.30).  
One can easily check that the corresponding  BIGs are not  $2$-e.c. 
\end{proof}

\subsection{$\{1\}$-Block intersection graphs of Steiner quadruple systems }
\label{sec:SQS2}

Before stating the lemmas, we need a few basic facts in design theory. 
Let $\scrD = (V, \calB)$ be a $t$-$(v, k, \lambda)$ design. 
Given $x \in V$, let $\calB_x$ denote the collection of all the blocks containing $x$ in $\calB$.
Then $\scrD_x := (V \setminus \{x\} , \{ B \setminus \{x\} : B \in \calB_x \})$ forms a
 $(t-1)$-$(v-1, k-1, \lambda)$ design, called a \emph{derived design} of  $\scrD$. 

\begin{lemma}\label{lemma:derived_big}
For any $x \in V$, 
the block intersection graph $\Gamma_{\scrD_x}$ of $\scrD_x$
is isomorphic to 
the induced subgraph $\Gamma_{\scrD}^{\bbN\setminus\{1\}}[\calB_x]$ 
of the $(\bbN\setminus\{1\})$-block intersection graph $\Gamma_{\scrD}^{\bbN\setminus\{1\}}$ 
of $\scrD$ on $\calB_x$.
\end{lemma}
\begin{proof}
There is a natural  bijection $B \mapsto B \setminus \{x\}$ between $\calB_x \subset \calB$
and the block set of $\scrD_x$ (i.e., the vertex set of $\Gamma_{\scrD_x}$).
For any distinct $B_1, B_2 \in \calB_x$, the pair $\{B_1, B_2\}$ forms an edge 
in $\Gamma_{\scrD}^{\bbN\setminus\{1\}}$ 
if and only if $\abs{B_1 \cap B_2} \ge 2$. 
Noting that $x \in B_1 \cap B_2$, it is equivalent to saying  
$(B_1 \setminus \{x\}) \cap  (B_2 \setminus \{x\}) \neq \emptyset$, 
which corresponds to an edge 
in $\Gamma_{\scrD_x}$; and vice versa for the non-adjacency. 
\end{proof}

\begin{theorem}
Let $\scrD$ be an SQS($v$).  
Then $\Xi(\Gamma^{\{1\}}_\scrD) \ge 2$ if and only if $v \ge 14$.
\end{theorem}
\begin{proof}
In an SQS$(v)$, any two blocks intersect at $0$, $1$ or $2$ point(s). 
Accordingly, 
the $\{1\}$-BIG of $\scrD$ is $2$-e.c.
 if and only if its $\{0,2\}$-BIG  (its complementary graph) is $2$-e.c.

First, we consider the $\{2\}$-BIG $\Gamma^{\{2\}}_\scrD$ of $\scrD$.
By Corollary~\ref{cor:STS_2ec}, any STS$(v-1)$ with $v \ge 14$ is $2$-e.c.
Furthermore, by Lemma~\ref{lemma:derived_big}, it is equivalent to saying 
 the induced subgraph $\Gamma^{\{2\}}_\scrD[\calB_x]$ is $2$-e.c. for each $x \in V$.
Since any two blocks in $\calB_x$ have non-empty intersection, 
   $\Gamma^{\{2\}}_\scrD[\calB_x]$ is  exactly the same as $\Gamma^{\{0, 2\}}_\scrD[\calB_x]$. 
In other words, for any distinct $B_1, B_2 \in \calB$ with $B_1 \cap B_2 \ne \emptyset$,
the $2$-e.c. property is satisfied in $\Gamma^{\{0,2\}}_\scrD$.

Next, we assume $v \ge 16$. Then $\scrD$ is a $2$-$(v, k, \lambda_2)$ design with 
$\lambda_2 = (v-2)/{2}  \ge 7$.
Let $B_1$ and $B_2$ be two disjoint blocks in $\calB$. 
Without loss of generality, we set $B_1 = \{1,2,3,4\}$ and $B_2= \{5,6,7,8\}$.
Proceeding similarly to the proof of Theorem~\ref{thm:BIG_PBD}, 
in order to show the $2$-e.c. property with respect to $B_1$ and $B_2$, 
it suffices to find three blocks which are adjacent to both $B_1$ and $B_2$,
to neither $B_1$ nor $B_2$, and to either $B_1$ or $B_2$, respectively.

Among all $\lambda_2 - 1 \ge 6$ blocks containing $\{1,2\}$ other than $B_1$, 
 there exist at least two blocks containing $\{1,2\}$ but disjoint from $B_2$. 
Let $Q_{20}$ be such a block. Then $\{Q_{20}, B_1\}$ and $\{Q_{20}, B_2\}$
are edges in $\Gamma^{\{2\}}_\scrD$ and $\Gamma^{\{0\}}_\scrD$, respectively.
 
Among all $\lambda_2 \ge 7$ blocks containing $\{1,5\}$, 
there are at most six blocks moreover containing a point in
$B_1 \cup B_2 \setminus \{1,5\}$.
Hence, there must exist a block, say $Q_{11}$, such that $Q_{11} \cap B_1 =\{1\}$ and
$Q_{11} \cap B_2 =\{5\}$. Then, in $\Gamma^{\{0,2\}}_\scrD$,  
$Q_{11}$ is not adjacent to $B_1$, as well as $B_2$. 

There are exactly six blocks of the form $\{1, a, b, h\}$ with $\{a,b\} \in \binom{B_2}{2}$.
Since $\scrD$ is an SQS, $h$ cannot be an element in $B_2$. Then, there are at least 
three blocks of the form $\{1, a, b, h\}$ with $\{a,b\} \in \binom{B_2}{2}$
and $h \notin B_1 \cup B_2$.   
Let $Q_{12}$ be such a block. 
Then $Q_{12}$  is adjacent to $B_2$ but   not adjacent to $B_1$ in  $\Gamma^{\{0,2\}}_\scrD$.

In summary, $\Gamma^{\{0,2\}}_\scrD$, as well as $\Gamma^{\{1\}}_\scrD$,  is $2$-e.c. if $v \ge 16$. 
Last, it is straightforward   
to verify that $\Gamma^{\{1\}}_\scrD$  is not $2$-e.c. if 
$\scrD$ is the unique SQS$(8)$ or  the unique SQS$(10)$. 
Moreover, 
one can check that the $\{1\}$-BIG  is $2$-e.c.  
for each of the four non-isomorphic SQS$(14)$. 
(See \cite{mendelsohn1972}, also \cite{hartman1992steiner,lindner1978steiner} and \cite{colbourn2007handbook} \S II.5.30 for the classification results). 
\end{proof}

To end this paper, we leave the following problem as an open question.

\begin{problem}
Find an upper bound on $v$ for the $\{1\}$-BIG of a $(v, K, \lambda)$-PBD to be $n$-e.c.
 with $n \ge 3$ (cf. Theorem~\ref{thm:3-e.c.-upper}).
Or give an upper bound on $n$ with respect to $K$ and $\lambda$
for an $n$-e.c. $\{1\}$-BIG (cf. \eqref{thm:n-e.c._basic_2} and \eqref{thm:n-e.c._basic_3} 
of Theorem~\ref{thm:n-e.c._basic}).  
\end{problem}

\bibliographystyle{abbrv}    
\bibliography{ECbib190225}


\end{document}